\documentclass[12pt]{article}

\usepackage{hyperref}

\usepackage{amsfonts, amsmath, amssymb, amsthm}
\usepackage{a4}

\usepackage{upgreek}
\usepackage{tikz-cd}

\DeclareMathOperator{\Ker}{Ker}

\DeclareMathOperator{\lk}{lk}
\DeclareMathOperator{\Gr}{Gr}
\newcommand{\tildetimes}{\mathbin{\widetilde{\times}}}
\newcommand{\strutik}{\vrule height 2.6ex depth 1ex width 0pt}

\newtheorem{theorem}{Theorem}

\theoremstyle{remark}
\newtheorem*{remark}{Remark}
\newtheorem{ir}{Important Remark}
\newtheorem*{xmp}{Example}

\title{Polynomial-valued constant hexagon cohomology}
\author{Igor G. Korepanov}

\date{April 2019}

\begin{document}

\sloppy

\maketitle

\begin{abstract}
Hexagon relations are algebraic realizations of four-dimen\-sional Pachner moves. `Constant'---not depending on a 4-simplex in a triangulation of a 4-manifold---hexagon relations are proposed, and their polynomial-valued cohomology is constructed. This cohomology yields polynomial mappings defined on the so called `coloring homology space', and these mappings can, in their turn, yield piecewise linear manifold invariants. These mappings are calculated explicitly for some examples.

It is also shown that `constant' hexagon relations can be obtained as a limit case of already known `nonconstant' relations, and the way of taking the limit is not unique. This non-uniqueness suggests the existence of an additional structure on the `constant' coloring homology space.
\end{abstract}

\section{Introduction}\label{s:i}

A piecewise linear (PL) manifold may be specified combinatorially by its given \emph{triangulation}. Most things we want to know about a PL manifold belong, however, to the manifold itself, and must be independent of a specific triangulation. This leads to the idea of, first, representing a transition from one triangulation to another as a combination of simple steps, and second, inventing an algebraic structure that corresponds to a triangulation but behaves under these steps in such a simple way that can produce quantities that do not change under these steps at all.

The mentioned simple steps are provided by the theorem of Pachner~\cite{Pachner,Lickorish}. In application to four-dimen\-sional closed PL manifolds, it states that there are just three kinds of \emph{Pachner moves} that can, together with their inverses, make up a chain connecting any two triangulations of a given manifold. Then, an algebraic structure is proposed based on \emph{hexagon relations} and their \emph{cohomology}.

Hexagon relations are `algebraic realizations' of four-dimen\-sional Pachner moves: this means, informally, that they imitate these moves algebraically. Somewhat similar things are known for \emph{three}-dimen\-sional manifolds and their quantum invariants~\cite{TV}. It appears, however, that \emph{quandles} and quandle cohomology---that yield invariants of \emph{knots} and their higher analogues~\cite{CKS}---are closer to our constructions.

Hexagon relations may be called analogues of quandles, and hexagon cohomology analogue of quandle cohomology, but applicable to 4-manifolds rather than knots. One essential difference is that, in the most general known case, hexagon relations (together with their cohomology) are \emph{not constant}: they vary for different 4-simplices (\emph{pentachora}) of the triangulation. In particular, a construction of invariants of a \emph{pair} ``manifold, middle cohomology class'' has been proposed in~\cite{nonconstant}, based on these nonconstant relations.

In the present paper, however, we work with \emph{constant} relations, and just a little bit with their neighborhood in the `nonconstant' space. The point is that constant relations are, first, already interesting in themselves, and second, calculations (see Subsection~\ref{ss:tn} below) suggest that there is a very nontrivial interplay between the `constant' and `nonconstant' cases.

The contents of the rest of this paper by sections is as follows:
\begin{itemize}\itemsep 0pt
 \item in Section~\ref{s:C}, we introduce `permitted colorings' of a simplicial complex---the basis for our hexagon relations. Interestingly, our permitted colorings immediately bring about some---their own---sort of homology~\eqref{ch};
 \item in Section~\ref{s:H}, we describe four-dimen\-sional Pachner moves in the form suitable for us, and introduce \emph{linear} constant hexagon relations;
 \item in Section~\ref{s:P}, we introduce our polynomial-valued hexagon cochain complex. This is a simple and general algebraic construction, dealing just with a sequence of  `standard' simplices, one for any dimension $n=1,2,\ldots$ (this may be contrasted with the `coloring homology' mentioned above which depends on a chosen simplicial complex!);
 \item in Section~\ref{s:K}, we discover the interplay between the coloring homology, hexagon cohomology, and usual simplicial cohomology: under some technical conditions, a hexagon cohomology class and a coloring homology class produce together a simplicial cohomology class~\eqref{er}. The dependence of the latter on the coloring homology class is, however, \emph{nonlinear}!
 \item In Section~\ref{s:M}, we specialize these results for four-\emph{manifolds}, introducing mappings \eqref{col-3} and~\eqref{col-4} that are \emph{independent of a triangulation};
 \item in Section~\ref{s:L}, we explain how our `constant' hexagon can be obtained from the `nonconstant' hexagon of~\cite{nonconstant}. This does not look completely trivial;
 \item in Section~\ref{s:E}, we present some calculation results showing what the mentioned mappings \eqref{col-3} and~\eqref{col-4}---from which PL manifold invariants can be extracted algebraically---can actually look like. Also, as we have already mentioned, we demonstrate a very nontrivial interplay between the `constant' and `nonconstant' cases.
 \item Finally, in Section~\ref{s:D}, we briefly discuss our results and further work.
\end{itemize}

\section{Permitted and edge-generated colorings of a simplicial complex}\label{s:C}

\subsection{Definition of a coloring}

A \emph{coloring} of a simplicial complex~$K$ means assigning a \emph{color} to each of its simplices of a given dimension~$n$. Color means here an element of a given \emph{color set}~$X$. In this paper, we take $n=3$, and $X=F^2$ ---a two-dimensional linear space over a fixed field~$F$. Thus, our coloring is a map
\begin{equation}\label{Kc}
(\text{set of all tetrahedra in }K) \to F^2.
\end{equation}
We write the color of an individual tetrahedron~$t$---its image under the map~\eqref{Kc}---as a two-column
\begin{equation}\label{sft}
\mathsf x_t = \begin{pmatrix} x_t \\ y_t \end{pmatrix},\qquad x_t, y_t \in F.
\end{equation}

\begin{remark}
Of course, coloring simplices of other dimension(s) than three may be also of interest, as well as using a set of colors other than~$F^2$. In this paper we, however, confine ourself to the case~\eqref{Kc}.
\end{remark}

\begin{remark}
Also, using ring~$\mathbb Z$ of integers instead of field~$F$ might be of interest.
\end{remark}

\subsection{Vertex ordering}\label{ss:vo}

Typically, we will be considering finite simplicial complexes~$K$ with vertices \emph{numbered} from 1 through their total number~$N_0^{(K)}$. The subscript here stays for the fact that vertices are zero-dimen\-sional; more generally, the number of $n$-sim\-plices in a finite simplicial complex will be denoted $N_n$ or~$N_n^{(K)}$. Note that our ``standard'' simplex~$\Delta^n$ has thus vertices $1,\ldots,n+1$, instead of probably more usual $0,\ldots,n$.

Consider, however, a situation where we deal with a complex~$\mathcal K$ and its subcomplex $K\subset \mathcal K$ with vertices
\begin{equation}\label{nmb}
i_1,\ldots, i_{N_0^{(K)}}, \qquad i_1<\ldots <i_{N_0^{(K)}},
\end{equation}
whose numbering is inherited from~$\mathcal K$ (for instance, $\mathcal K=\Delta^n$ may be an $n$-simplex, and $K$ one of its $(n-1)$-faces). In a situation like this, $K$ can retain its vertex numbering~\eqref{nmb}, but the point is that if we have proved in this paper a theorem about a complex~$K$ whose vertices have numbers $1,\ldots, N_0^{(K)}$, then it can be always transferred to the same complex with vertices denoted as in~\eqref{nmb}, just by the obvious substitution
\begin{equation}\label{nmb-subst}
1\to i_1,\quad\ldots\;,\quad N_0^{(K)}\to i_{N_0^{(K)}}.
\end{equation}

We usually denote triangles by the letter~$s$, tetrahedra by the letter~$t$, and pentachora by the letter~$u$. When we write these in terms of their vertices, these latter go, by default, in the \emph{increasing order}:
\begin{equation*}
\text{if}\quad t=ijkl, \quad\text{then}\quad i<j<k<l.
\end{equation*}

\subsection{Permitted colorings of one pentachoron}

Interesting structures appear when we declare some of the colorings \emph{permitted} (which means of course that other colorings are prohibited). Our permitted colorings will form a \emph{linear subspace} in the space of all colorings, described in terms of either \emph{edge functionals} or \emph{edge vectors}. Some motivation for introducing these functionals and vectors can be found in~\cite{nonconstant} (and see also Section~\ref{s:L} of the present work for explanation of how the structures introduced below can be obtained from those in~\cite{nonconstant}).

In this Subsection, we begin with introducing our permitted colorings for just one pentachoron.

\subsubsection*{Permitted colorings in terms of edge functionals}

In this approach, permitted colorings of a pentachoron are singled out by \emph{linear relations}. Namely, there is one linear relation associated with each pentachoron \emph{edge}~$ij$, formulated as the vanishing of a linear \emph{edge functional}~$\upphi_{ij}$. This~$\upphi_{ij}$ can depend only on the colors of the three tetrahedra containing the edge: $t\supset ij$. Edge functionals are defined for \emph{unoriented} edges: $\upphi_{ij}=\upphi_{ji}$.

The set (linear space)~$V_u$ of permitted colorings for a pentachoron~$u$ is, by definition, the intersection of kernels of all ten edge functionals:
\begin{equation}\label{pk}
V_u = \bigcap_{ij\subset u} \Ker \upphi_{ij}.
\end{equation}

The colorings of a tetrahedron~$t$ being written as two-columns~\eqref{sft}, we can write the restriction of~$\upphi_{ij}$ onto~$t$, or \emph{$t$-component} of~$\upphi_{ij}$, as a two-\emph{row}:
\begin{equation}\label{phi_ij-t}
\upphi_{ij}|_t = \begin{pmatrix} \phi_{t,ij}^{(1)} & \phi_{t,ij}^{(2)} \end{pmatrix}.
\end{equation}

Consider pentachoron $u=12345$ and its 3-face $t=1\ldots \hat i\ldots 5$---that is, tetrahedron~$t$ lies \emph{opposite} vertex~$i$. The $t$-components of (nonvanishing on~$t$) edge functionals are, by definition, as follows:
\begin{equation}\label{phi_1234-const}
\setlength\arraycolsep{.6em}
\begin{pmatrix} \upphi_{k_1k_2} \\ \upphi_{k_1k_3} \\ \upphi_{k_1k_4} \\
 \upphi_{k_2k_3} \\ \upphi_{k_2k_4} \\ \upphi_{k_3k_4} \end{pmatrix}_t =
(-1)^{i+1}
\begin{pmatrix}0 & 1\\
1 & -1\\
-1 & 0\\
-1 & 0\\
1 & 1\\
0 & -1\end{pmatrix},
\end{equation}
where $1\le k_1<k_2<k_3<k_4\le 5$ are the four vertices of~$t$, and we have written just one subscript~$t$ meaning $t$-component for all~$\upphi_{ij}$. For other pentachora $i_1i_2i_3i_4i_5$, substitution~\eqref{nmb-subst} applies, that is, $1\to i_1$, \ldots, $5\to i_5$.

A direct calculation shows that
\begin{equation}\label{du}
\dim V_u = 5.
\end{equation}

\begin{xmp}
Relation $\upphi_{12}=0$ in pentachoron~$12345$ looks as follows:
\begin{equation}\label{12}
y_{1234} - y_{1235} + y_{1245} = 0.
\end{equation}
\end{xmp}

\subsubsection*{Permitted colorings in terms of edge vectors}

The same linear space~$V_u$ of permitted colorings of one pentachoron can be described as the span of ten \emph{edge vectors}. Given an edge~$b$, edge vector~$\uppsi_b$ is a permitted coloring of a simplicial complex---at this moment, one pentachoron---that has nonvanishing components only for tetrahedra $t\supset b$. Namely, by definition, for tetrahedron~$1234$ they are as follows:
\begin{equation}\label{psi_1234-const}
\begin{pmatrix} \uppsi_{12} & \uppsi_{13} & \uppsi_{14} & \uppsi_{23} & \uppsi_{24} & \uppsi_{34} \end{pmatrix}_{1234} 
 = \begin{pmatrix}1 & -1 & 0 & 0 & 1 & -1\\
0 & 1 & -1 & -1 & 1 & 0\end{pmatrix},
\end{equation}
while for other pentachora $i_1i_2i_3i_4$, substitution~\eqref{nmb-subst} again applies, that is, $1\to i_1$, \ldots, $4\to i_4$.

\begin{ir}
There are no additional signs in~\eqref{psi_1234-const}, in contrast with $(-1)^{i+1}$ in~\eqref{phi_1234-const}!
\end{ir}

The fact that edge vectors~\eqref{psi_1234-const} generate the same five-dimen\-sional space~$V_u$ as in~\eqref{pk}, is checked by a direct calculation.

\subsection{Permitted and edge-generated colorings of a simplicial complex}\label{ss:pe}

By definition, a permitted coloring of a simplicial complex~$K$ is such whose restriction onto any pentachoron~$u\subset K$ is permitted. Permitted colorings of~$K$ form a linear space denoted~$V_K$.

And \emph{edge-generated colorings} are, also by definition, linear combinations of edge vectors~$\uppsi_b$ whose components are described by the same formula~\eqref{psi_1234-const} as for one pentachoron (but there may be of course more than three tetrahedra containing a given edge in an arbitrary~$K$). Edge-generated colorings form a linear \emph{subspace}
\begin{equation}\label{pe}
V_K^{(0)} \subset V_K
\end{equation}
in the space of permitted colorings, because edge vectors generate, according to the above, permitted colorings in every pentachoron.

\subsection{Coloring homology}\label{ss:ch}

In this paper, we will understand \emph{coloring homology} simply as the factor
\begin{equation}\label{ch}
H_{\mathrm{col}}(K,F)\, \stackrel{\mathrm{def}}{=}\, V_K / V_K^{(0)} .
\end{equation}
This is the only homology group of the following very short chain complex:
\begin{equation}\label{cc}
F^{N_1} \xrightarrow{\mbox{\scriptsize edge vectors}} F^{2 N_3}\xrightarrow{ \mbox{\scriptsize \begin{tabular}{c}edge functionals\\in each pentachoron\end{tabular} } } F^{10 N_4} .
\end{equation}
Recall (see the beginning of Subsection~\ref{ss:vo}) that $N_n=N_n^{(K)}$ is the number of $n$-simplices in complex~$K$. The first term in~\eqref{cc} consists of formal linear combinations of edges with coefficients in~$F$. The second term is the space of \emph{all} (permitted or not) colorings of~$K$, and each edge~$b$ is sent by the first arrow to the edge vector~$\uppsi_b$. Finally, the second arrow is the direct sum of \emph{all} edge functionals; these act in each pentachoron separately, and in each pentachoron there are ten of them. The rightmost term in~\eqref{cc} is thus the direct sum of copies of~$F$ for each pair $u\supset b$, with $u$ a pentachoron and $b$ and edge.

\begin{ir}
Sequence \eqref{cc}, or a modification of it, can actually be extended both to the left and to the right, compare~\cite[\color{black}Sections 5.2 and~5.3]{nonconstant}. This means that more coloring, or `exotic', homology groups can be defined. In the present paper we, however, work only with $H_{\mathrm{col}}(K,F)$ given by~\eqref{ch}.
\end{ir}

\section{Four-dimensional Pachner moves and linear constant hexagon relations}\label{s:H}

\subsection{Pachner moves}\label{ss:P}

Consider a 5-simplex $\Delta^5$. Its boundary~$\partial \Delta^5$ consists of six pentachora (\,=\,4-simplices). Imagine that $k$ of these pentachora, $1\le k\le 5$, enter in a triangulation of a four-dimen\-sional piecewise linear (PL) manifold~$M$. Then we can replace them with the remaining $6-k$ pentachora, without changing~$M$. This is called four-dimen\-sional \emph{Pachner move}, and there are five kinds of them: 1--5, 2--4, 3--3, 4--2, and 5--1; here the number before the dash is~$k$, while the number after the dash is, of course, $6-k$.

We sometimes call the initial configuration---cluster of~$k$ pentachora---the \emph{left-hand side} (l.h.s.)\ of the Pachner move, while its final configuration---cluster of~$6-k$ pentachora---its \emph{right-hand side} (r.h.s.).

\subsection{Linear constant hexagon relations}\label{ss:fh}

Consider one of the Pachner moves as described in Subsection~\ref{ss:P}, and denote its l.h.s.\ as $C$, and its r.h.s.\ as~$\bar C$.

\begin{ir}\label{ir:full}
According to Subsection~\ref{ss:vo}, there is a given \emph{order} on the vertices of our simplex~$\Delta^5$, for instance, $\Delta^5=123456$. So, it must be noted that $C$ is allowed to consist of \emph{any} $k$ pentachora, whatever the numbers of their vertices. Informally, we use the words \emph{full hexagon} for combinatorial or algebraic statements relating $C$ and~$\bar C$ in such situation. This applies, in particular, to the following Theorem~\ref{th:fh}.
\end{ir}

\begin{theorem}\label{th:fh}
 \begin{enumerate}
  \item\label{i:fh1} The restrictions of permitted colorings of the left-hand side~$C$ of a Pachner move (described above) onto the common boundary $\partial C=\partial \bar C$ yield the same set of colorings of this common boundary as the restrictions of permitted colorings of the right-hand side~$\bar C$.
  \item\label{i:fh2} Moreover, all these permitted colorings of $\partial C=\partial \bar C$ can be generated by edge vectors\/~$\uppsi_b$ for edges $b\subset \partial C=\partial \bar C$ only.
  \item\label{i:fh3} There are fixed numbers~$a_k$, \ $1\le k\le 5$, such that, if we fix the zero coloring on $\partial C=\partial \bar C$, the dimension of permitted coloring space of (inner tetrahedra of)~$C$ is~$a_k$, and the same dimension for~$\bar C$ is~$a_{6-k}$. Namely,
\begin{equation*}
a_1 = a_2 = a_3 = 0, \qquad a_4 = 1, \qquad a_5 = 3.
\end{equation*}
These permitted colorings of only inner tetrahedra are generated by vectors\/~$\uppsi_b$ for only inner edges\/~$b$ in either~$C$ or~$\bar C$.
  \item\label{i:fh4} Consider a situation where\/ $C$ was, and\/ $\bar C$ has become, part of triangulation of a PL 4-manifold~$M$. Denote\/ $K_{\mathrm{ini}}$ and\/~$K_{\mathrm{fin}}$ the simplicial complexes determined by the corresponding triangulations of~$M$. Then, there is a canonical isomorphism (see Subsection~\ref{ss:pe} for notations)
\begin{equation}\label{i-f}
H_{\mathrm{col}}(K_{\mathrm{ini}},F) \,\cong\, H_{\mathrm{col}}(K_{\mathrm{fin}},F)
\end{equation}
that can be described as follows. We say that a permitted coloring of\/~$K_{\mathrm{ini}}$ \emph{corresponds} to a permitted coloring of\/~$K_{\mathrm{fin}}$ if their restrictions onto the closed complement of\/~$C$ or\/~$\bar C$ coincide. This correspondence is not generally one-to-one, but it becomes an isomorphism after factoring by edge-generated colorings.
 \end{enumerate}
\end{theorem}

\begin{proof}
Items \ref{i:fh1}--\ref{i:fh3} are proved by direct calculations. Item~\ref{i:fh4} follows then from the fact that the contribution of inner---with respect to~$C$---edges is the same in $V_{K_{\mathrm{ini}}}$ and $V_{K_{\mathrm{ini}}^{(0)}}$, and the same applies if we change $C$ to~$\bar C$ and subscript `$\mathrm{ini}$' to `$\mathrm{fin}$'.
\end{proof}

\begin{remark}
It may make sense to remind once again that a motivation for the above constructions can be found in~\cite{nonconstant}, combined with Section~\ref{s:L} of the present paper.
\end{remark}

\section{Polynomial-valued hexagon cochain complex: Definition}\label{s:P}

By definition, a \emph{constant polynomial $n$-cochain}~$\mathfrak c$, for $n\ge 3$, is an arbitrary polynomial defined on the linear space $V_{\Delta^n}$ of all permitted colorings of the standard $n$-simplex $\Delta^n = 1\dots (n{+}1)$.

The \emph{coboundary}~$\delta \mathfrak c$ of~$\mathfrak c$ is the polynomial
\begin{equation}\label{cb}
\delta \mathfrak c = \sum_{k=1}^{n+2} (-1)^{k+1}\, \mathfrak c_{1\dots \widehat{k} \dots (n+2)} 
\end{equation}
defined on the linear space $V_{\Delta^{n+1}}$ of all permitted colorings of the standard $(n{+}1)$-simplex $\Delta^{n+1} = 1\dots (n{+}2)$. In~\eqref{cb}, each $n$-face $1\dots \widehat{k} \dots (n+2)$ of~$\Delta^{n+1}$ is identified with the standard~ $\Delta^n$ in the natural way---that is, according to the general rule~\eqref{nmb-subst}.

The complex can be written as follows:
\begin{equation}\label{c}
0\longrightarrow C_{\mathrm{hex}}^3 \stackrel{\delta}{\longrightarrow} C_{\mathrm{hex}}^4 \stackrel{\delta}{\longrightarrow} \dots\,,
\end{equation}
where $C_{\mathrm{hex}}^n$ means the linear space of $n$-cochains.

\paragraph{Bilinear cochains.} One simple variation on the theme of polynomial-valued hexagon cochain complex may also be of interest. Namely, we define a \emph{constant bilinear $n$-cochain} as a bilinear form
\begin{equation}\label{bi}
V_{\Delta^n}\times V_{\Delta^n} \to F.
\end{equation}
This means that now a \emph{pair} of permitted colorings comes into play. Definition~\eqref{cb} remains the same in this case, as well as the form~\eqref{c} of the complex.

Our reasonings below in Sections \ref{s:K} and~\ref{s:M} will apply to both the `polynomial' and `bilinear' cases. We will prefer, however, to formulate and prove our Theorems \ref{th:f}, \ref{th:c} and~\ref{th:m} first for polynomial-valued hexagon cochain complexes as defined in the beginning of this Section, and then point out the changed necessary for the `bilinear' case after their respective proofs. Hopefully, this will make our exposition less cumbersome.

\section{Hexagon cohomology, coloring homology, and simplicial cohomology in a simplicial complex}\label{s:K}

In this Section, we work within a fixed finite simplicial complex~$K$, with the numbering of its vertices also fixed.

\subsection{A permitted coloring produces a chain map from hexagon polynomial cochain complex to simplicial cochain complex}\label{ss:f}

Standard $n$-simplex $\Delta^n = 1\dots (n{+}1)$ is isomorphic to any $n$-simplex $\sigma^n\subset K$. To be exact, we will be working with the isomorphism conserving the order of vertices, according to our general rule~\eqref{nmb-subst}. This isomorphism yields, in particular, the (bijective) mapping
\begin{equation}\label{Ds}
(\text{set of tetrahedra in }\Delta^n) \to (\text{set of tetrahedra in }\sigma^n).
\end{equation}

A permitted coloring of $\Delta^n$ or~$\sigma^n$ is a mapping from the l.h.s.\ or r.h.s.\ of~\eqref{Ds}, respectively, to~$F^2$. Hence, \eqref{Ds} yields the mapping of these colorings, and in the \emph{opposite} direction:
\begin{equation}\label{cDs}
(\text{permitted colorings of }\Delta^n) \leftarrow (\text{permitted colorings of }\sigma^n).
\end{equation}

Recalling now the polynomial cochain definition given in the beginning of Section~\ref{s:P}, and taking into account that our polynomials can be understood as $F$-valued \emph{functions}, we come to the following mapping (again from left to right):
\begin{equation}\label{oDs}
C_{\mathrm{hex}}^n
\to ( F \text{-valued functions on permitted colorings of }\sigma^n),
\end{equation}
where $C_{\mathrm{hex}}^n$ means the space of polynomial hexagon $n$-cochains, see~\eqref{c}. Taking mappings~\eqref{oDs} for all $\sigma^n \subset K$ at once, and \emph{assuming that a permitted coloring of~$K$ is given}, we arrive finally at a linear mapping
\begin{equation}\label{fn}
f_n\colon\quad C_{\mathrm{hex}}^n \to C^n(K,F),
\end{equation}
where $C^n(K,F)$ is the linear space of usual simplicial $F$-valued $n$-cochains on~$K$.

To justify the header of this Subsection, it remains to prove the following theorem.

\begin{theorem}\label{th:f}
Given a permitted coloring of~$K$, mappings~$f_n$~\eqref{fn} form together a chain map~$f$. That is, the following diagram
\begin{equation}\label{d}
\begin{tikzcd}
0 \arrow[r] \arrow[d]
 & C_{\mathrm{hex}}^3 \arrow[r, "\delta"] \arrow[d, "f_3"]
  & C_{\mathrm{hex}}^4 \arrow[r, "\delta"] \arrow[d, "f_4"]
   & \ldots \\
C^2(K,F) \arrow[r]
 & C^3(K,F) \arrow[r]
  & C^4(K,F) \arrow[r]
   & \ldots
\end{tikzcd}
\end{equation}
is commutative. Here the first row is the complex~\eqref{c}, while the second row is a truncated simplicial cochain complex for~$K$, with 0- and 1-cochain spaces cut out.
\end{theorem}

\begin{proof}
It remains to note that the consistency of~$f$ with the codifferentials in the two complexes follows at once from the fact that these codifferentials are given by the alternated sums of the same kind~\eqref{cb} over $(n-1)$-faces of~$\Delta^n$.
\end{proof}

\begin{ir}
For our polynomial cocycles, chain map~$f$ depends on a permitted coloring also \emph{polynomially}.
\end{ir}

\paragraph{Bilinear case.} The constructions of this Subsection, including Theorem~\ref{th:f}, are easily transferred to the bilinear case (see the paragraph containing~\eqref{bi}). In this case, a \emph{pair} of permitted colorings must be given, and the first row in~\eqref{d} will consist of bilinear hexagon cochains.

\subsection{Induced cohomology map and coloring homology}

Chain map~$f$~\eqref{d} induces mappings of the cohomology spaces. Since $f$ depends on a chosen permitted coloring, so do the cohomology mappings. In some important cases a cohomology mapping depends actually only on the \emph{coloring homology class} of the permitted coloring, and some---although not all!---of such cases are covered by the following theorem.

Recall (see, for instance, \cite{Lickorish}) that the \emph{link} $\lk(A,K)$ of a simplex~$A$ in a simplicial complex~$K$ consists of all simplices~$B$ such that the \emph{join} $A\star B$ is also a simplex in~$K$.

\begin{theorem}\label{th:c}
Suppose that, for every edge~$b$ of complex~$K$, the link\/ $\lk(b,K)$ has trivial $(n{-}2)$-th simplicial cohomology group:
\begin{equation}\label{lk}
H^{n-2}(\lk(b,K), F) = 0.
\end{equation}
Then, the cohomology map induced by chain map~$f$~\eqref{d} in dimension~$n\ge 3$,
\begin{equation}\label{cm}
f^{(n)}\colon\quad H_{\mathrm{hex}}^n \to H^n(K,F),
\end{equation}
depends only on the coloring homology class of the permitted coloring.
\end{theorem}

\begin{remark}
The situation where map~\eqref{cm} depends on a coloring homology element can of course be formulated as a map
\begin{equation}\label{er}
H_{\mathrm{hex}}^n \times H_{\mathrm{col}}(K,F) \to H^n(K,F).
\end{equation}
\end{remark}

\begin{proof}
It is enough to consider the change of the permitted coloring by \emph{one} edge vector, with a coefficient from~$F$, and corresponding to an edge~$b$. We have to prove that the image $\phi=f(c)$ of a hexagon cocycle~$c$ under~$f$ can change only by a simplicial coboundary. We denote the `old' and `new' versions of~$\phi$ as $\phi_{\mathrm{old}}$ and~$\phi_{\mathrm{new}}$, and their difference as
\[
\Delta = \phi_{\mathrm{new}} - \phi_{\mathrm{old}}.
\]
As $c$ is a cocycle and $f$ a chain map, all of $\phi_{\mathrm{old}}$, $\phi_{\mathrm{new}}$ and~$\Delta$ are \emph{simplicial cocycles}, and we must show that $\Delta$ is, moreover, a \emph{coboundary}.

In our situation, $\phi$ changes only \emph{locally}, namely, only on simplices $\sigma^n$ containing~$b$ and representable thus as
\begin{equation}\label{ts}
\sigma^n=b\star \tau^{n-2},
\end{equation}
where simplex~$\tau^{n-2}$ belongs to~$\lk(b,K)$. Consider a simplicial $(n{-}2)$-cochain~$\Xi$ on~$\lk(b,K)$ defined by
\begin{equation}\label{wg}
\Xi(\tau^{n-2}) = \Delta(\sigma^n)
\end{equation}
for $\tau^{n-2}$ and~$\sigma^n$ as in~\eqref{ts}. As $\Delta$ is a cocycle, $\Xi$ is also a \emph{cocycle on the link\/}~$\lk(b,K)$ and, due to~\eqref{lk}, also a \emph{coboundary}
\begin{equation}\label{wh}
\Xi = \delta \, \widehat{\Xi}
\end{equation}
of an $(n{-}3)$-cochain~$\widehat{\Xi}$ on~$\lk(b,K)$.

We now define $(n-1)$-cochain~$\widehat{\Delta}$ on~$K$, vanishing outside~$\lk(b,K)$, as follows:
\begin{equation}\label{wD}
\widehat{\Delta}(\sigma^{n-1}) = \begin{cases} \widehat{\Xi}(\tau^{n-3}) & \text{for }\sigma^{n-1} = b\star \tau^{n-3},\\
0 & \text{for other }\sigma^{n-1}. \end{cases}
\end{equation}
It follows from \eqref{wg}, \eqref{wh} and~\eqref{wD} that $\Delta = \delta \widehat{\Delta}$.
\end{proof}

\paragraph{Bilinear case.} The bilinear version of Theorem~\ref{th:c} says that the cohomology maps~\eqref{cm} depend on the \emph{pair} of coloring homology classes of the permitted colorings taking part in the construction. As for the proof, it begins with the words `it is enough to consider the change of \emph{one} permitted coloring by \emph{one} edge vector'. Otherwise, everything goes the same way as in the `polynomial' case.

\section{Hexagon cohomology, coloring homology, and simplicial cohomology in a piecewise linear four-manifold}\label{s:M}

Theorem~\ref{th:c} can be reformulated (especially if we look at~\eqref{er}) as follows: given a hexagon $n$-cocycle~$c$ (up to a coboundary), and if the technical condition~\eqref{lk} is fulfilled, we obtain a polynomial mapping
\begin{equation}\label{col-n}
g_{\mathrm{col}}^{(n)} \colon \quad H_{\mathrm{col}}(K,F) \to H^n(K,F).
\end{equation}

Recall that we worked in Section~\ref{s:K} within a fixed finite simplicial complex~$K$, with the
numbering of its vertices also fixed. Now we are going to consider a \emph{four-dimen\-sional piecewise linear manifold}~$M$, and let $K$ represent a triangulation of~$M$ (this is often written as $M=|K|$). As the link of an edge in a PL 4-mani\-fold is a 2-sphere, \eqref{lk} holds of course for $n=3$, but not for $n=4$. Nevertheless, mapping~\eqref{col-n} can be defined quite naturally for $n=4$ as well, using the following roundabout way.

Let $I=[0,1]$, and consider the direct product $\mathcal M=M\times I$. In~$\mathcal M$, the link of every edge has trivial second cohomology, hence, Theorem~\ref{th:c} does work for $\mathcal M$ and $n=4$. Also, $M\times \{0\} \cong M$ is a deformation retract of~$\mathcal M$, so there exists a canonical isomorphism $H^4(\mathcal M,F) \cong H^4(M,F)$.

The conclusion is that the cohomology map~\eqref{cm}, for $n=4$ and $K$ a triangulation of~$M$, again depends only on the coloring homology class of the permitted coloring. Or, in other words, for a hexagon 4-cocycle given to within adding a coboundary, and $n=4$, we have also defined the polynomial mapping~\eqref{col-n}.

Moreover, we are going to show that mapping~\eqref{col-n}, for $n=$ either~$3$ or~$4$, does \emph{not} actually depend on the specific triangulation of~$M$ or numbering of its vertices. As there is no problem with replacing $H^n(K,F)$ with $H^n(M,F)$, we will focus our attention on the term $H_{\mathrm{col}}(K,F)$ and on the mapping~$g_{\mathrm{col}}^{(n)}$ itself.

\begin{theorem}\label{th:m}
Let $K_1$ and~$K_2$ be two simplicial complexes corresponding to two triangulations of a given PL 4-manifold~$M$, and let either $n=3$ or $n=4$. Then, there exists an \emph{isomorphism} $\iota\colon\; H_{\mathrm{col}}(K_1,F)\to H_{\mathrm{col}}(K_2,F)$ making the following diagram commutative:
 \begin{equation}\label{D}
  \begin{tikzcd}[row sep=tiny]
H_{\mathrm{col}}(K_1,F) \arrow[dr] \arrow[dd, "\iota"] &  \\
& H^n(M,F) \\
H_{\mathrm{col}}(K_2,F) \arrow[ur] & 
  \end{tikzcd}
 \end{equation}
Two non-vertical arrows in~\eqref{D} are of course the versions of~$g_{\mathrm{col}}^{(n)}$ for $K_1$ and~$K_2$.
\end{theorem}

The proof of Theorem~\ref{th:m} is given in the next three subsections. Namely, in Subsections \ref{ss:c3} and~\ref{ss:c4}, we consider the case where the triangulation is changed by one Pachner move, for $n=3$ and $n=4$, respectively. We must pay attention to the vertex order, so we emphasize that our Pachner moves are as described in Subsection~\ref{ss:P} (don't forget also Important Remark~\ref{ir:full}), together with the possibility~\eqref{nmb-subst} of changing the vertex numbering without changing the order. As any two triangulations can be connected by a chain of Pachner moves, it will remain to show how the vertex order \emph{can} be changed, and this is done in Subsection~\ref{ss:vn}.

\subsection{Dimension three}\label{ss:c3}

Let $n=3$, and consider what happens with mapping~\eqref{col-n} under a Pachner move. Returning to the notations of Theorem~\ref{th:fh}, we denote $K_{\mathrm{ini}}$ and~$K_{\mathrm{fin}}$ the triangulations of~$M$ before and after this move, and use the isomorphism~\eqref{i-f} between the coloring homology groups.

On the other hand, it is known that the $\frown$-product between $n$-th cohomology and $n$-th homology with coefficients in a field is nondegenerate. This implies that, in our situation, any element~$h$ of the third cohomology group is determined by the values that any 3-cocycle representing~$h$ takes on 3-\emph{cycles} modulo 3-\emph{boundaries}. We can now take all simplicial 3-cycles intended for calculating the mentioned values such that they contain no \emph{inner} tetrahedra of cluster $C$ or~$\bar C$ (that is, with zero coefficients at such tetrahedra. For notations $C$ or~$\bar C$ see Subsection~\ref{ss:fh}). We can thus describe mapping~\eqref{col-n} \emph{avoiding} the cluster changed by the Pachner move. This means, together with item~\ref{i:fh4} of Theorem~\ref{th:fh}, that the commutative diagram~\eqref{D}, for $K_1=K_{\mathrm{ini}}$, \ $K_2=K_{\mathrm{fin}}$, and \eqref{i-f} as~$\iota$, indeed takes place.

\subsection{Dimension four}\label{ss:c4}

We now consider how a Pachner move affects the image of a given element $h_{\mathrm{col}}\in H_{\mathrm{col}}(K,F)$ under mapping~\eqref{col-n}, for $n=4$. This Pachner move replaces (using the notations of Section~\ref{s:H}) a cluster~$C$ of $k$~pentachora, $1\le k\le 5$, by~$\bar C$---its closed complement in~$\partial \Delta^5$.

Recall that there is a hexagon 4-\emph{cocycle}~$c$ behind our mapping~\eqref{col-n}. This~$c$ is sent to a simplicial cocycle by~$f_4$ in diagram~\eqref{d}, and $f_4(c)$ is nothing but a representative of the simplicial cohomology element $g_{\mathrm{col}}^{(4)}(h_{\mathrm{col}})$. In application to~$\partial \Delta^5$, this means that
\begin{equation}\label{gD}
g_{\mathrm{col}}^{(4)}(h_{\mathrm{col}})\frown \partial \Delta^5 = 0.
\end{equation}

Now, if our PL manifold~$M$ is orientable, we proceed as follows. Choose an orientation of~$M$; it induces also an orientation of either $C$ or~$\bar C$ as part of~$M$. These orientations of $C$ and~$\bar C$ are, however, \emph{not consistent} if $C$ and~$\bar C$ are regarded as part of~$\partial \Delta^5$. Hence, if $C$ and~$\bar C$ are oriented this way, \eqref{gD} implies
\begin{equation}\label{CbC}
g_{\mathrm{col}}^{(4)}(h_{\mathrm{col}})\frown C - g_{\mathrm{col}}^{(4)}(h_{\mathrm{col}})\frown \bar C = 0.
\end{equation}
We see that replacing $C$ by~$\bar C$ simply does not change the product $g_{\mathrm{col}}^{(4)}(h_{\mathrm{col}})\frown c$ for any 4-\emph{cycle} $c\in H_4(M,F)$, and hence the image of~$h_{\mathrm{col}}$ in~$H^4(M,F)$ stays also the same.

In the most general case, $M$ may consist of several connected components. For the orientable components, we proceed as above. A non-orientable component yields a nontrivial 4-cycle only if our field~$F$ is of characteristic~2, in which case \eqref{gD} also surely implies~\eqref{CbC}.

\subsection{Independence of vertex numbering}\label{ss:vn}

We have shown the invariance of mappings \eqref{col-n}, for $n=3$ and $n=4$, under Pachner moves. One small problem that still remains is that we were always assuming that the vertices of any triangulation are \emph{ordered} (see Subsection~\ref{ss:vo}). We are now going to show that these mappings do not actually depend on the order of vertices.

Indeed, here is how we can change the position of any one vertex~$v$ in this ordering. We do any chain of Pachner moves that removes~$v$ from the triangulation; this removal is of course performed by a move 5--1. Then we do all this chain backwards, but when doing the corresponding move 1--5, we change the position of~$v$ in the order of vertices into any other we like. Such possibility is of course ensured by the fact that we have a \emph{full} hexagon, see Important Remark~\ref{ir:full}.

We have thus proved Theorem~\ref{th:m}, and defined the following mappings for a PL 4-manifold~$M$:
\begin{equation}\label{col-3}
H_{\mathrm{col}}(M,F) \to H^3(M,F), \text{ \ \ for a given hexagon 3-cocycle},
\end{equation}
and
\begin{equation}\label{col-4}
H_{\mathrm{col}}(M,F) \to H^4(M,F), \text{ \ \ for a given hexagon 4-cocycle}.
\end{equation}

\paragraph{Bilinear case.} In the bilinear case, $H_{\mathrm{col}}(K,F)$ in~\eqref{col-n} is replaced by $H_{\mathrm{col}}(K,F)\times H_{\mathrm{col}}(K,F)$. Similarly, the spaces of color homologies are replaced with their Cartesian squares in the formulation of Theorem~\ref{th:m} and in its proof (namely, in Subsection~\ref{ss:c4}). Otherwise, everything goes the same way, and we arrive at the bilinear versions of \eqref{col-3} and~\eqref{col-4}; these can be found below as \eqref{col-bil-3} and~\eqref{col-bil-4}.

\section{Constant hexagon as a limiting case of nonconstant hexagon}\label{s:L}

Our `constant' edge functionals~\eqref{phi_1234-const} and edge vectors~\eqref{psi_1234-const} can be obtained as a limiting case (formal limit in the case of a finite characteristic) of `nonconstant' edge functionals and edge vectors introduced in~\cite{nonconstant}. We will content ourself here with explaining how it works only for edge \emph{functionals}; for edge \emph{vectors}, the procedure is much the same and is left as an easy exercise for the reader. Also, we will do everything on the example of one tetrahedron $t=1234$.

We first take edge functionals in the form~{\color{black}\cite[(31)]{nonconstant}}:
\begin{equation}\label{phi_1234-inf}
\setlength\arraycolsep{.6em}
\begin{pmatrix} \upphi_{12} \\ \upphi_{13} \\ \upphi_{14} \\
 \upphi_{23} \\ \upphi_{24} \\ \upphi_{34} \end{pmatrix}_{1234} =
\begin{pmatrix} \omega_{234}-\omega_{134} & 0 \\
                \omega_{124} & \omega_{234} \\
                -\omega_{123} & -\omega_{234} \\
                -\omega_{124} & -\omega_{134} \\
                \omega_{123} & \omega_{134} \\
                0 & \omega_{123}-\omega_{124}
\end{pmatrix} ,
\end{equation}
where $\omega$ is a $F$-valued simplicial 2-cocycle (which means, in application to the tetrahedron~$1234$, that its values entering~\eqref{phi_1234-inf} satisfy $\omega_{123}-\omega_{124}+\omega_{134}-\omega_{234}=0$). We choose our~$\omega$ as follows:
\begin{equation}\label{om-rho}
\omega_{ijk} = 1+o\cdot\varrho_{ijk}, \qquad i<j<k,
\end{equation}
where $\varrho$ is a \emph{given} simplicial 2-cocycle, and $o$ is an infinitesimal parameter. To be exact, $o$ is finite at this moment, but we are going to set $o\to 0$. Also, $\varrho$ is supposed to be generic enough so that we don't encounter division by zero in our expressions below (see expression for~$A_2$ in~\eqref{A12}).

Now, we compose new edge functionals~$\upphi'_{ij}$ as follows: denote
\begin{equation}\label{Ao}
A_o = \begin{pmatrix} 1 & o^{-1} \\ 0 & -o^{-1} \end{pmatrix}
\end{equation}
---this matrix will be responsible for the invertible linear transformation in the space $V_t = F^2$ of colorings of our tetrahedron~$t$ for each finite~$o$---and set
\begin{equation}\label{f_1234-dinf}
\setlength\arraycolsep{.6em}
\begin{pmatrix} \upphi'_{12} \\ \upphi'_{13} \\ \upphi'_{14} \\
 \upphi'_{23} \\ \upphi'_{24} \\ \upphi'_{34} \end{pmatrix}_{1234} =
\lim_{o\to 0}\, \biggl(
\begin{pmatrix} \upphi_{12} \\ \upphi_{13} \\ \upphi_{14} \\
 \upphi_{23} \\ \upphi_{24} \\ \upphi_{34} \end{pmatrix}_{1234} \!\!\! A_o \biggr) =
\begin{pmatrix} 0 & \varrho_{234}-\varrho_{134} \\
                1 & \varrho_{124}-\varrho_{234} \\
                -1 & -\varrho_{123}+\varrho_{234} \\
                -1 & -\varrho_{124}+\varrho_{134} \\
                1 & \varrho_{123}-\varrho_{134} \\
                0 & -\varrho_{123}+\varrho_{124}
\end{pmatrix} .
\end{equation}

Then we make the linear transform in the space of tetrahedron~$1234$ colorings, corresponding to multiplying \eqref{f_1234-dinf} from the right by the product $A_1 A_2$, where
\begin{equation}\label{A12}
A_1 = \begin{pmatrix} 1 & \varrho_{134}- \varrho_{124}\\ 0 & 1\end{pmatrix},\qquad\quad
A_2 = \begin{pmatrix} 1 & 0\\ 0 & (\varrho_{123}- \varrho_{124})^{-1}\end{pmatrix},
\end{equation}
that is, we set
\begin{equation*}
\begin{pmatrix} \upphi''_{12} \\ \upphi''_{13} \\ \upphi''_{14} \\
 \upphi''_{23} \\ \upphi''_{24} \\ \upphi''_{34} \end{pmatrix}_{1234} =
\begin{pmatrix} \upphi'_{12} \\ \upphi'_{13} \\ \upphi'_{14} \\
 \upphi'_{23} \\ \upphi'_{24} \\ \upphi'_{34} \end{pmatrix}_{1234} \!\!\! A_1 A_2 .
\end{equation*}
Finally, we \emph{rename} $\upphi''_{ij}\mapsto \upphi_{ij}$. We have arrived exactly at~\eqref{phi_1234-const}, for $k_1=1$, \ldots, $k_4=4$.

\begin{ir}\label{ir:b}
We could of course apply the linear transformation corresponding to the whole product $A_o A_1 A_2$ \emph{before} taking the limit $o\to 0$. We hope, however, that our step-by-step approach makes things clearer.
\end{ir}

\begin{ir}\label{ir:f}
As we see, the limiting process described above is far from being unique, because it depends on a chosen cocycle~$\varrho$.
\end{ir}

\section{Experimental results}\label{s:E}

We present here, just for illustration, some calculation results showing how mappings \eqref{col-3} and~\eqref{col-4}, or their bilinear analogues, can look explicitly. Then, in Subsection~\ref{ss:tn}, we briefly tell the reader what happens with the spaces $V_K$ and~$V_K^{(0)}$---recall formula~\eqref{ch}---at the point of passing from constant to nonconstant hexagon. The conclusion is that, first, our `constant' case is already very intriguing, and second, that the investigation of its neighborhood within the `nonconstant' case---which will be a big and separate research---is extremely promising.

\subsection{Some nontrivial hexagon cocycles}\label{ss:ntc}

\paragraph{Bilinear 3-cocycle.} Given a nontrivial hexagon bilinear 3-cocycle, we come to the bilinear analogue of~\eqref{col-3}, that is, a bilinear mapping
\begin{equation}\label{col-bil-3}
H_{\mathrm{col}}\times H_{\mathrm{col}} \to H^3(M,F).
\end{equation}
Namely, we will use the following 3-cocycle:
\begin{equation}\label{c(3)}
c^{(3)} = -x y' - y x'.
\end{equation}
Here $\begin{pmatrix}x \\ y\end{pmatrix}$ and $\begin{pmatrix}x' \\ y'\end{pmatrix}$ are a pair colorings of the tetrahedron~$1234$, as required in~\eqref{bi}. Recall that \emph{all} colorings of a separate tetrahedron are permitted.

\paragraph{Bilinear 4-cocycle.} Similarly, there is also the following nontrivial hexagon bilinear 4-cocycle:
\begin{equation}\label{c(4)}
c^{(4)} = y_{2345}\, y'_{1234}.
\end{equation}
which yields a mapping
\begin{equation}\label{col-bil-4}
H_{\mathrm{col}}\times H_{\mathrm{col}} \to H^4(M,F).
\end{equation}
---the bilinear analogue of~\eqref{col-4}. Recall that cocycle~\eqref{c(4)}---as well as \eqref{c1(4)} and \eqref{c2(4)} below---belongs to the `standard' pentachoron $\Delta^4 = 12345$, see Section~\ref{s:P}.

\paragraph{Two cubic 4-cocycles in characteristic 2.} There are two linearly independent modulo coboundaries \emph{cubic} hexagon 4-cocycles in characteristic 2. The first of them is obtained from~\eqref{c(4)} by setting $y'_t=y_t^2$:
\begin{equation}\label{c1(4)}
c_1^{(4)} = y_{2345}\, y_{1234}^2
\end{equation}
Recall that raising to the second power is a linear operation---\emph{Frobenius endomorphism}---for a field of characteristic~2.

To make the structure of the second 4-cocycle more visible, we introduce, for the moment, the following notations:
\begin{align*}
a &= x_{2345} + y_{2345}, \quad
b = x_{1345} + y_{1345}, \\
c &= x_{1245} + y_{1245}, \quad
d = x_{1235} + y_{1235}, \quad
e = x_{1234} + y_{1234}.
\end{align*}
The cocycle is then
\begin{equation}\label{c2(4)}
c_2^{(4)} = bde + bce + ace + acd + abd.
\end{equation}

\paragraph{Other nontrivial cocycles in finite characteristics.} Many more nontrivial cocycles have been calculated in~\cite{cubic}. Note that a different basis in the two-dimensional space of colorings of one tetrahedron was used in~\cite{cubic}. Namely, if we denote, for a moment, the $x$ and~$y$ of~\cite{cubic} as $\tilde x$ and~$\tilde y$ (and our colors~\eqref{sft} as simply $x$ and~$y$), then
\begin{equation*}
\begin{pmatrix}\tilde x \\ \tilde y\end{pmatrix} = \begin{pmatrix}-1 & 1 \\ 1 & 0\end{pmatrix} \begin{pmatrix}x \\ y\end{pmatrix} .
\end{equation*}

\subsection{Calculations for specific manifolds using constant cohomology}\label{ss:m}

The first experimental result, and unexplained as yet, is that the dimension~$d$ of coloring homology space~$H_{\mathrm{col}}(M,F)$ is the sum of the dimensions of two usual cohomology groups:
\begin{equation}\label{nb}
d \stackrel{\mathrm{def}}{=} \dim H_{\mathrm{col}}(M,F) = \dim H^2(M,F) + \dim H^3(M,F).
\end{equation}
We introduce some basis in~$H_{\mathrm{col}}(M,F)$, and write polynomial functions on~$H_{\mathrm{col}}(M,F)$ in terms of coordinates $X_1,\ldots,X_d$ w.r.t.\ this basis. 
For the bilinear case, the coordinates of the second element in~$H_{\mathrm{col}}(M,F)$ are denoted by primed letters $X'_1,\ldots,X'_d$.

For mapping \eqref{col-3}, we also introduce a basis in~$H^3(M,F)$. Hence, mapping \eqref{col-3}, corresponding to cocycle~$c^{(3)}$~\eqref{c(3)}, is given by $d_3=\dim H^3(M,F)$ polynomials; we denote them below as $p_1^{(3)},\ldots,p_{d_3}^{(3)}$.

For mapping \eqref{col-4}, we identify an element of~$H^4(M,F)$ with its value on the fundamental class~$[M]$ (because we are going to consider only connected manifolds). Hence, any of cocycles $c^{(4)}$, $c_1^{(4)}$ and~$c_2^{(4)}$ of Subsection~\ref{ss:ntc} gives just one polynomial, denoted respectively as $p^{(4)}$, $q^{(4)}$ and~$r^{(4)}$.

Below our notations are as usual: $S^n$ is an $n$-dimen\-sional sphere, $T^n$ is an $n$-dimen\-sional torus, $\mathbb RP^n$ is an $n$-dimen\-sional real projective space, $\mathbb CP^2$ is a \emph{complex} two-dimen\-sional projective space, and $S^2\tildetimes S^2$ denotes the \emph{twisted} product of two spheres~$S^2$.

We present our results in the following form: the manifold~$M$ and the field~$F$ we are working with form a header highlighted by underlining, and then go the experimental results for them. We think it is enough to give here a few examples with just one field $F=\mathbb F_2$---the prime Galois field of two elements; the more so because $\mathbb F_2$ works well with both orientable and unorientable manifolds.

\begin{ir}
Experimental result~\eqref{nb} works of course in all characteristics, as far as we could check.
\end{ir}

\paragraph{$\underline{M=\mathbb CP^2,\qquad F=\mathbb F_2 \strutik}$}
\begin{align*}
d &= 1,\\
p^{(4)} &= X_1 X'_1,\\
q^{(4)} &= X_1^3,\\
r^{(4)} &= q^{(4)}.
\end{align*}

\paragraph{$\underline{M=S^2\times S^2,\qquad F=\mathbb F_2 \strutik}$}
\begin{align*}
d &= 2,\\
p^{(4)} &= X_1 X'_2+X_2 X'_1,\\
q^{(4)} &= X_1^2 X_2+X_1 X_2^2,\\
r^{(4)} &= q^{(4)}.
\end{align*}

\paragraph{$\underline{M=S^2\tildetimes S^2,\qquad F=\mathbb F_2 \strutik}$}
\begin{align*}
d &= 2,\\
p^{(4)} &= X_1 X'_2+X_2 X'_1+X_2 X'_2,\\
q^{(4)} &= X_1^2 X_2+X_1 X_2^2+X_2^3,\\
r^{(4)} &= q^{(4)}.
\end{align*}

\paragraph{$\underline{M=S^2\times T^2,\qquad F=\mathbb F_2 \strutik}$}
\begin{align*}
d &= 4,\\
p_1^{(3)} &= X_1 X'_3+X'_1 X_3,\\
p_2^{(3)} &= X_1 X'_4+X'_1 X_4,\\[1.5ex]
p^{(4)} &= X_1 X'_2+X_1 X'_3+X_2 X'_1+X'_1 X_3,\\
q^{(4)} &= X_1^2 X_2+X_1^2 X_3+X_1 X_2^2+X_1 X_3^2,\\
r^{(4)} &= q^{(4)}.
\end{align*}

\paragraph{$\underline{M=\mathbb RP^2\times S^2,\qquad F=\mathbb F_2 \strutik}$}
\begin{align*}
d &= 3,\\
p_1^{(3)} &= X_2 X'_3+X'_2 X_3,\\[1.5ex]
p^{(4)} &= X_1 X'_3+X'_1 X_3,\\
q^{(4)} &= X_1^2 X_3+X_1 X_3^2,\\
r^{(4)} &= X_1^2 X_3+X_1 X_3^2+X_2^2 X_3.
\end{align*}

\paragraph{$\underline{M=\mathbb RP^2\times T^2,\qquad F=\mathbb F_2 \strutik}$}
\begin{align*}
d &= 7,\\
p_1^{(3)} &= X_1 X'_6+X'_1 X_6+X_4 X'_5+X_5 X'_4+X_5 X'_6+X_6 X'_5,\\
p_2^{(3)} &= X_1 X'_2+X_1 X'_6+X_1 X'_7+X_2 X'_1+X'_1 X_6+X'_1 X_7+X_3 X'_4\\
 &\qquad +X_3 X'_6+X_4 X'_3+X_4 X'_7+X_6 X'_3+X_6 X'_7+X_7 X'_4+X_7 X'_6,\\
p_3^{(3)} &= X_2 X'_5+X'_2 X_5+X_3 X'_6+X_4 X'_7+X_5 X'_6+X_5 X'_7+X_6 X'_3\\
 &\qquad +X_6 X'_5+X_7 X'_4+X_7 X'_5,\\[1.5ex]
p^{(4)} &= X_1 X'_7+X'_1 X_7+X_3 X'_5+X_5 X'_3+X_5 X'_7+X_7 X'_5,\\
q^{(4)} &= X_1^2 X_7+X_1 X_7^2+X_3^2 X_5+X_3 X_5^2+X_5^2 X_7+X_5 X_7^2,\\
r^{(4)} &= X_1^2 X_7+X_1 X_7^2+X_2 X_5^2+X_3^2 X_5+X_3^2 X_6+X_3 X_5^2+X_4^2 X_7\\
 &\qquad +X_5^2 X_6+X_5 X_7^2+X_6^2 X_7+X_6 X_7^2.
\end{align*}

\paragraph{$\underline{M=\mathbb RP^2\times \mathbb RP^2,\qquad F=\mathbb F_2 \strutik}$}
\begin{align*}
d &= 5,\\
p_1^{(3)} &= X_1 X'_5+X'_1 X_5+X_1 X'_4+X'_1 X_4+X_2 X'_3+X'_2 X_3,\\
p_2^{(3)} &= X_3 X'_5+X_2 X'_5+X'_3 X_5+X'_2 X_5+X_3 X'_4+X'_3 X_4,\\[1.5ex]
p^{(4)} &= X_1 X'_5+X'_1 X_5+X_3 X'_3,\\
q^{(4)} &= X_1 X_5^2+X_1^2 X_5+X_3^3,\\
r^{(4)} &= X_3^2 X_5+X_2^2 X_5+X_1^2 X_5+X_1 X_4^2+X_3^2 X_4+X_3^3+X_2 X_3^2.
\end{align*}

\paragraph{$\underline{M=\mathbb RP^4,\qquad F=\mathbb F_2 \strutik}$}
\begin{align*}
d &= 2,\\
p_1^{(3)} &= X_1 X'_2+X_2 X'_1,\\[1.5ex]
p^{(4)} &= X_2 X'_2,\\
q^{(4)} &= X_2^3,\\
r^{(4)} &= X_1^2 X_2.
\end{align*}

In the above examples, one can see that
\begin{equation}\label{qr}
q^{(4)} = r^{(4)}
\end{equation}
for all \emph{orientable} manifolds. Interestingly, equality~\eqref{qr} may be violated for more complicated manifolds. Namely, define the simplest \emph{twisted tori} as follows. First, we denote~$\tilde T_n^3$ the fiber bundle with base~$S^1$, fiber~$T^2$, and monodromy matrix~$\begin{pmatrix}1&n\\ 0&1\end{pmatrix}$. Such fiber bundles are \emph{three-dimen\-sional} twisted tori. Then we consider \emph{four-dimen\-sional} twisted tori~$\tilde T_n^4$ defined simply as direct products of~$\tilde T_n^3$ with a circle:
\begin{equation*}
\tilde T_n^4 = \tilde T_n^3 \times S^1.
\end{equation*}
If $n=0$, we get of course the usual torus $\tilde T_0^4 = T^4$.

The unexpected calculation result is: \eqref{qr} holds for $n=0$, 1, 3 and~4, but \emph{not} for $n=2$.

\subsection{What happens at the point of passing to a nonconstant case}\label{ss:tn}

In the `nonconstant' case of paper~\cite{nonconstant}, there are also linear spaces $W_K$ of permitted and $W_K^{(0)}$ of edge-generated colorings---direct analogues of spaces $V_K$ and~$V_K^{(0)}$ introduced in Subsection~\ref{ss:pe}. A very interesting question is what happens with them under the passage to the limit described in Section~\ref{s:L}. This is going to be the subject of a separate research; here we only explain some simple ideas and inform the reader of some experimental facts. Also, we restrict ourself here to the field $F=\mathbb R$ of real numbers, in order to be able to use simple analytic arguments.

Space~$W_K^{(0)}$ is the linear span of $N_1$ vectors in~$F^{2N_3}=\mathbb R^{2N_3}$ (see the first paragraph in Subsection~\ref{ss:vo} for notations). Suppose we pass to the limit in such way that $W_K^{(0)}$ retains a constant dimension~$d^{(0)}$, then it has a limiting location (at least one, due to the compactness of the real Grassmannian $\Gr(d^{(0)},\mathbb R^{2N_3})$)---a subspace $L_K^{(0)}\subset \mathbb R^{2N_3}$ of the same dimension~$d^{(0)}$.

\begin{remark}
The prelimit invertible linear transform---direct sum of transforms given by the product of matrices \eqref{Ao} and~\eqref{A12} (see also Important Remark~\ref{ir:b}), taken over all two-dimen\-sional color spaces~$V_t\cong F^2$ for all tetrahedra~$t$---clearly does not affect the validity of this argument.
\end{remark}

Every separate `nonconstant' edge vector also has its `constant' limit, given by~\eqref{psi_1234-const} (recall the first paragraph of Section~\ref{s:L}). The linear span of these limiting edge vectors is nothing but our `constant' space~$V_K^{(0)}$, and of course $V_K^{(0)} \subset L_K^{(0)}$, but it may---and does---happen that this inclusion is strict!

Space~$W_K$ is the opposite case: it is singled out by linear \emph{restrictions}. So, its limit~$L_K$ may only be \emph{smaller} than the `constant' space~$V_K$. We come this way to the chain of inclusions:
\begin{equation}\label{coi}
V_K^{(0)} \subset L_K^{(0)} \subset L_K \subset V_K.
\end{equation}

Now the experimental facts:
\begin{enumerate}\itemsep 0pt
 \item at least sometimes,
\begin{equation}\label{L0V0}
\dim L_K^{(0)} = \dim V_K^{(0)}+1,
\end{equation}
 \item at least sometimes,
\begin{equation}\label{LV}
\dim L_K = \dim V_K - 2k, \qquad k=1,2,\ldots.
\end{equation}
\end{enumerate}

\section{Discussion}\label{s:D}

Polynomials calculated in Subsection~\ref{ss:m}, if taken as they are, are not PL manifold invariants, because their definition requires a \emph{basis} in the coloring homology space~$H_{\mathrm{col}}(M,F)$. Moreover, a basis in $H^3(M,F)$ is also required for polynomials $p_1^{(3)},\ldots,p_{d_3}^{(3)}$. Invariant is, of course, the set of these polynomials taken up to linear transformations of the mentioned bases---but this is not the point where to stop!

\begin{remark}
In paper~\cite{cubic}, some simple invariants \emph{were} actually calculated. Namely, given one of our polynomials $q^{(4)}$, $r^{(4)}$, or $q^{(4)}+r^{(4)}$, we let its variables take values from a finite \emph{extension}~$\mathbb F_{2^k}$ of field~$\mathbb F_2$ and calculated, for each $v\in \mathbb F_{2^k}$, how many times the polynomial takes value~$v$. Polynomials $p_1^{(3)},\ldots,p_{d_3}^{(3)}$ were not considered in~\cite{cubic}.
\end{remark}

What looks much more interesting is the fact that the chain~\eqref{coi} of embeddings, taking into account experimental facts \eqref{L0V0} and~\eqref{LV}, brings about some additional structure on our `constant' space~$V_K$ and hence---see~\eqref{ch}---on the coloring homology space, probably relating this latter to usual cohomologies. Given the existence of a great many nontrivial hexagon cocycles~\cite{cubic}, this may lead to very interesting consequences.

\smallskip

Finally, it must be said that the constructions proposed in \cite{cubic,nonconstant} and the present paper, are not confined to just four-dimen\-sional manifolds. Similar things surely can be done in three dimensions, based, for example, on the pentagon relations proposed in~\cite{pentagon}. Moving in the opposite direction, there are indications of the existence of interesting \emph{heptagon} relations for five dimensions.

\end{document}